\documentclass[12pt]{article}

\usepackage[utf8]{inputenc}
\usepackage[T1]{fontenc}
\usepackage{amsmath, amssymb, amsthm}
\usepackage{authblk}

\usepackage{mathrsfs}

\usepackage{xurl}
\usepackage{hyperref}

\providecommand{\subjclass}[2][]{%
  \par\smallskip\noindent\textbf{MSC 2020. }#2\par\smallskip
}
\providecommand{\keywords}[1]{%
  \par\smallskip\noindent\textbf{Keywords. }#1\par\smallskip
}

\title{\bf Structural Chirality from Inverse Semigroups to Twisted Groupoid $C^*$-Algebras}

\author{\Large Takao Inou\'{e}}

\affil{\large Faculty of Informatics, Yamato University, \\ Osaka, Japan\footnote{Email: inoue.takao@yamato-u.ac.jp; \\ Personal Email: takaoapple@gmail.com \\ [I prefer my personal email address for correspondence.]}}

\date{February 26, 2026}

\theoremstyle{definition}
\newtheorem{definition}{Definition}[section]

\theoremstyle{plain}
\newtheorem{proposition}[definition]{Proposition}
\newtheorem{lemma}[definition]{Lemma}
\newtheorem{theorem}[definition]{Theorem}

\theoremstyle{remark}
\newtheorem{remark}[definition]{Remark}

\newcommand{\Iso}{\mathrm{Iso}}
\newcommand{\Mir}{\mathrm{Mir}}
\newcommand{\Atp}{\mathrm{Atp}}

\begin{document}

\maketitle

\begin{abstract}
We develop a structural theory of chirality for inverse semigroups and show how it
propagates canonically to \'{e}tale groupoids and twisted groupoid $C^*$-algebras.
Starting from inverse semigroup data equipped with admissible twist information,
we construct a canonical twisted universal groupoid in the sense of Paterson and
introduce a mirror correspondence encoding intrinsic asymmetry.
Our main result identifies a structural obstruction to self-oppositeness (mirror self-duality) at the
level of twisted universal groupoids and shows that this obstruction descends to an
obstruction for the associated reduced twisted groupoid $C^*$-algebra to be
isomorphic to its opposite.
The framework is representation-independent, yet compatible with concrete germ
groupoid models, and provides a unified bridge between partial symmetries,
groupoid structures, and analytic invariants in noncommutative operator algebras.
\end{abstract}

\keywords{inverse semigroup, universal groupoid, \'{e}tale groupoid, twisted groupoid $C^*$-algebra, 
structural chirality, self-oppositeness, mirror self-duality}

\subjclass[2020]{Primary 20M18; Secondary 22A22, 46L05, 46L55}

\tableofcontents

\section{Introduction}

Symmetry under reversal or mirroring is a ubiquitous theme across algebra,
geometry, and operator algebras.
In many settings, however, such symmetry is not automatic: algebraic or
dynamical structures may fail to admit a canonical identification with their
mirror images.
This phenomenon, which we refer to as \emph{structural chirality}, arises when
the absence of mirror symmetry is not caused by accidental choices or external
constraints, but is instead encoded intrinsically in the algebraic structure
itself.
Understanding when such chirality appears, how it propagates across different
levels of abstraction, and how it manifests in analytic invariants is the
central motivation of this work.

The purpose of this paper is to develop a unified framework for structural
chirality that starts at the level of inverse semigroups, passes canonically
through \'{e}tale groupoids, and culminates in an obstruction to mirror
self-duality of twisted groupoid $C^*$-algebras.
Our main result shows that suitable inverse semigroup data give rise to a 
canonical twisted universal groupoid whose mirror symmetry properties are 
determined entirely by the underlying semigroup structure.
In particular, the existence or nonexistence of a mirror-isomorphism at the
inverse semigroup level is transported functorially to the groupoid level and
descends further to a precise criterion for when the associated reduced twisted
groupoid $C^*$-algebra is isomorphic to its opposite.
This establishes structural chirality as a robust invariant, independent of
auxiliary representations yet computable via concrete germ groupoid models.

\paragraph{Why universal groupoids?}
A crucial methodological choice in this paper is the use of the universal
groupoid associated with an inverse semigroup, rather than a germ groupoid
arising from a particular representation.
While germ groupoids provide concrete and often computationally convenient
models, they depend intrinsically on auxiliary representation data and may
therefore obscure structural phenomena that are representation-independent.
In contrast, the universal groupoid constructed from the idempotent structure
of an inverse semigroup is canonical and functorial, capturing the intrinsic
partial symmetry encoded by the semigroup itself.
This makes the universal groupoid the natural arena in which to formulate and
detect structural chirality.
Concrete germ groupoid models remain essential for explicit calculations and
examples; accordingly, their role is addressed separately in Appendix~A, where
we show that the universal and germ constructions are compatible and that the
chirality obstruction identified in the universal setting can be computed in
representation-level models.

\section{Related Work}

The present paper continues a line of investigation in which ``chirality'' is treated
as a \emph{structural} phenomenon governed by a groupoid of symmetries and
transport along equivalences.
Our point of departure is the groupoid-based formulation of chirality and racemization
for loops developed in \cite{InoueLoop}, which isolates mirror transitions as morphisms
to a mirror object and studies the resulting obstruction at the level of isotopy classes.
The guiding principle is that mirror phenomena should not be imposed externally, but
should be generated and transported functorially by intrinsic symmetries.
The current work extends this perspective to the world of partial symmetries, where
inverse semigroups and \'{e}tale groupoids provide the canonical language.

Inverse semigroups have long been understood as algebraic encodings of partial
symmetries, and their relationship with groupoids and operator algebras is a central
theme in noncommutative topology.
A comprehensive account of inverse semigroup structure is provided in Lawson's
monograph \cite{Lawson1998}.
Paterson's monograph \cite{Paterson1999} systematically develops the bridge between
inverse semigroups, universal \'{e}tale groupoids, and associated operator algebras,
and supplies the canonical groupoid construction that we adopt as the backbone of the
present paper.
From the $C^*$-algebraic viewpoint, Exel's survey \cite{Exel2008} gives a unified
framework connecting inverse semigroup actions, partial crossed products, and
combinatorial $C^*$-algebras, and provides motivation for treating inverse semigroup
data as a natural source of groupoid models.

On the groupoid side, Renault's foundational monograph \cite{Renault1980} introduced
the groupoid approach to $C^*$-algebras, establishing groupoids as organizing objects
for convolution constructions and representation theory.
The \'{e}tale setting is particularly well suited to a structural approach because
it permits a bisection-based calculus and often supports explicit constructions of
cocycles and twists; see Sims' lecture notes \cite{Sims2017} for a detailed guide.
Twists (or, equivalently in many cases, $\mathbb T$-valued $2$-cocycles) play a
distinguished role in encoding ``orientation'' data and in producing families of
non-isomorphic $C^*$-algebras attached to a fixed underlying groupoid; this is the
natural level at which self-oppositeness (mirror self-duality) becomes a nontrivial problem.

The connection between inverse semigroups and canonical \'{e}tale groupoids is also
prominent in the literature on Steinberg algebras and their operator-algebraic
counterparts.
In particular, Steinberg \cite{Steinberg2010} emphasizes the role of Paterson's
universal groupoid in describing convolution algebras associated to inverse semigroups,
highlighting the universal groupoid as a representation-independent object.
From a different angle, Sieben \cite{Sieben1997} develops $C^*$-crossed products by
partial actions and actions of inverse semigroups, which provides a natural analytic
framework complementary to the topological groupoid construction used here.

Relative to these works, our contribution is not to introduce a new class of groupoids
or $C^*$-algebras per se, but to propose a \emph{structural chirality invariant}
attached to inverse semigroup data and to show how it propagates canonically through
the universal groupoid construction and into the associated twisted groupoid $C^*$-algebra.
The novelty lies in organizing mirror symmetry, its failure, and a racemization-type
two-state dynamics within a single functorial package:
chirality is defined at the level of semigroup data, transported to the universal
twisted \'{e}tale groupoid, and then detected analytically as an obstruction to
self-oppositeness (mirror self-duality) of the resulting reduced twisted groupoid
$C^*$-algebra.
Concrete computations and representation-level intuition remain important; accordingly,
Appendix~A records the germ groupoid model attached to a specific faithful
representation and explains its compatibility with the universal setting.

\paragraph{Organization of the paper.}
Section~2 reviews the relevant literature and situates the present work within
inverse semigroup and groupoid approaches to operator algebras.
Section~3 provides a semigroup-level warm-up, formulating mirror constructions
and chirality invariants in a simplified isotopy setting.
Section~4 introduces represented inverse semigroups and their mirror
representations, establishing the basic transport and descent mechanism.
Section~5 develops the corresponding theory for twisted \'{e}tale groupoids and
connects mirror symmetry to opposite-algebra phenomena.
Section~6 adopts Paterson's universal groupoid as the canonical bridge from
inverse semigroup data to twisted \'{e}tale groupoids, and formulates the
universal mirror correspondence.
Section~7 states and proves the Main Theorem, showing that the mirror obstruction
propagates canonically from inverse semigroups to twisted universal groupoids and
descends to an obstruction to self-oppositeness of the associated reduced twisted
groupoid $C^*$-algebras.
Section~8 concludes with future directions.
Appendix~A presents the germ groupoid model attached to a faithful representation
and explains its compatibility with the universal construction.

\section{A Semigroup Warm-up: Decorated Semigroups and Isotopy}

\begin{remark}
This section serves as a warm-up, isolating the mirror and isotopy mechanism
in a simplified semigroup setting before passing to inverse semigroups and
canonical universal groupoids.
\end{remark}

\subsection{Decorated Semigroups and Isotopisms}

\begin{definition}[Decorated semigroup]
A \emph{decorated semigroup} is a pair $(S,\Sigma)$ where
$S$ is a semigroup and $\Sigma \subseteq S$ is a distinguished
finite generating subset.
\end{definition}

The decoration prevents automatic symmetry under reversal
of multiplication by fixing a structural orientation.

\begin{definition}[Isotopism]
Let $S$ and $T$ be semigroups.
An \emph{isotopism} from $S$ to $T$ is a triple of bijections
$(\alpha,\beta,\gamma): S \to T$
such that
\[
\alpha(x)\beta(y) = \gamma(xy)
\quad\text{for all } x,y\in S.
\]
\end{definition}

We denote by
\[
\Iso_{\mathscr S}(S,T)
\]
the set of isotopisms between $S$ and $T$.

\begin{definition}[Isotopy groupoid]
Objects are decorated semigroups.
Morphisms are isotopisms preserving the decoration:
\[
\alpha(\Sigma)=\beta(\Sigma)=\gamma(\Sigma).
\]
\end{definition}

---

\subsection{Mirror Semigroup}

\begin{definition}[Mirror semigroup]
Let $(S,\Sigma)$ be a decorated semigroup.
Its \emph{mirror} is
\[
(S,\Sigma)^\# := (S^{\mathrm{op}},\Sigma),
\]
where multiplication is reversed:
\[
x \cdot_\# y := yx.
\]
\end{definition}

\begin{definition}[Mirror-isotopisms]
\[
\Mir(S,\Sigma)
:=
\Iso_{\mathscr S}\bigl((S,\Sigma),(S,\Sigma)^\#\bigr).
\]
\end{definition}

Elements of $\Mir(S,\Sigma)$ are called
\emph{mirror-isotopisms}.

---

\subsection{Transport of Mirror-Isotopisms}

\begin{lemma}[Transport]\label{lem:transport_semigroup}
Let $h:(S,\Sigma)\to (T,\Lambda)$ be an isotopism.
Then
\[
h\,\Mir(S,\Sigma)\,h^{-1}
=
\Mir(T,\Lambda).
\]
\end{lemma}

\begin{proof}
If $(\alpha,\beta,\gamma)\in\Mir(S,\Sigma)$,
then by definition
\[
\alpha(x)\beta(y)=\gamma(yx).
\]
Conjugating by $h$ preserves the isotopy relation
and reverses multiplication in the same manner,
so the image defines a mirror-isotopism of $(T,\Lambda)$.
\end{proof}

---

\subsection{Chirality Index}

Assume $\Mir(S,\Sigma)$ is finite.

\begin{definition}[Chirality index]
\[
k([S,\Sigma])
=
\sum_{g\in\Mir(S,\Sigma)} w(g),
\]
where $w(g)>0$ are fixed weights.
\end{definition}

\begin{theorem}[Structural vanishing criterion]
The following are equivalent:
\begin{enumerate}
\item $k([S,\Sigma])=0$,
\item $\Mir(S,\Sigma)=\varnothing$,
\item there exists no isotopism
$(S,\Sigma)\to (S,\Sigma)^\#$.
\end{enumerate}
\end{theorem}

\begin{proof}
Since all weights are strictly positive,
the weighted sum vanishes iff the index set is empty.
The equivalence of (2) and (3) is immediate from the definition.
\end{proof}

---

\subsection{Intrinsic Symmetry and Descent}

Let
\[
\Atp(S,\Sigma)
=
\Iso_{\mathscr S}((S,\Sigma),(S,\Sigma)).
\]

We restrict admissible mirror transitions
to those generated by intrinsic autotopisms.

\begin{definition}[Intrinsic mirror dynamics]
A mirror transition is admissible iff
it is generated by $\Atp(S,\Sigma)$.
\end{definition}

\begin{theorem}[Descent to isotopy classes]
Under intrinsic mirror dynamics,
the chirality index $k([S,\Sigma])$
depends only on the isotopy class.
\end{theorem}

\begin{proof}
By Lemma~\ref{lem:transport_semigroup},
mirror-isotopisms are transported functorially
along isotopisms.
Since admissible transitions are generated internally,
the weighted count is invariant under isotopy.
\end{proof}

\section{Structural Chirality for Inverse Semigroups with Representation}

\subsection{Represented Inverse Semigroups}

\begin{definition}[Represented inverse semigroup]
A \emph{represented inverse semigroup} is a pair $(S,\rho)$ where
$S$ is an inverse semigroup and
\[
\rho:S\to I(X)
\]
is a faithful representation into the symmetric inverse semigroup
of partial bijections on a set $X$.
\end{definition}

The representation $\rho$ is regarded as fixed structural data.
In particular, isotopies are required to respect $\rho$.

\begin{remark}
By the Wagner--Preston theorem, every inverse semigroup
admits such a faithful representation.
Here we treat $\rho$ as part of the structure rather than as
an auxiliary construction.
\end{remark}

---

\subsection{Isotopisms of Represented Inverse Semigroups}

\begin{definition}[Isotopism]
Let $(S,\rho)$ and $(T,\sigma)$ be represented inverse semigroups.
An \emph{isotopism} from $(S,\rho)$ to $(T,\sigma)$
is a triple of bijections
\[
(\alpha,\beta,\gamma): S\to T
\]
such that
\[
\alpha(x)\beta(y)=\gamma(xy)
\quad\text{for all }x,y\in S,
\]
and the representations are intertwined:
\[
\sigma\circ\alpha=\sigma\circ\beta=\sigma\circ\gamma
\quad\text{on }\rho(S).
\]
\end{definition}

We write
\[
\Iso_{\mathscr I}((S,\rho),(T,\sigma))
\]
for the set of such isotopisms.

---

\subsection{Mirror Representation}

\begin{definition}[Mirror representation]
Let $(S,\rho)$ be a represented inverse semigroup.
The \emph{mirror representation} is defined by
\[
\rho^\#(s):=\rho(s^*)^{-1},
\]
where inversion is taken in $I(X)$.
\end{definition}

\begin{definition}[Mirror inverse semigroup]
The \emph{mirror} of $(S,\rho)$ is
\[
(S,\rho)^\#:=(S,\rho^\#).
\]
\end{definition}

\begin{remark}
Although $s\mapsto s^*$ provides a formal anti-isomorphism
of inverse semigroups,
the mirror representation $\rho^\#$ need not be isotopic
to $\rho$.
This asymmetry is the source of structural chirality.
\end{remark}

---

\subsection{Mirror-Isotopisms}

\begin{definition}[Mirror-isotopisms]
\[
\Mir(S,\rho)
:=
\Iso_{\mathscr I}\bigl((S,\rho),(S,\rho)^\#\bigr).
\]
\end{definition}

Elements of $\Mir(S,\rho)$ are called
\emph{mirror-isotopisms}.

---

\subsection{Transport of Mirror-Isotopisms}

\begin{lemma}[Transport]\label{lem:transport_inverse}
Let
\[
h:(S,\rho)\to(T,\sigma)
\]
be an isotopism.
Then
\[
h\,\Mir(S,\rho)\,h^{-1}
=
\Mir(T,\sigma).
\]
\end{lemma}

\begin{proof}
Conjugation by $h$ preserves the isotopy condition
$\alpha(x)\beta(y)=\gamma(xy)$.
Since $h$ intertwines the representations,
the induced conjugation also intertwines the mirror representations,
yielding the claimed equality.
\end{proof}

---

\subsection{Chirality Index}

Assume $\Mir(S,\rho)$ is finite.

\begin{definition}[Chirality index]
\[
k([S,\rho])
=
\sum_{g\in\Mir(S,\rho)} w(g),
\]
where all weights $w(g)>0$.
\end{definition}

\begin{theorem}[Structural vanishing criterion]
The following are equivalent:
\begin{enumerate}
\item $k([S,\rho])=0$,
\item $\Mir(S,\rho)=\varnothing$,
\item there exists no isotopism
$(S,\rho)\to (S,\rho)^\#$.
\end{enumerate}
\end{theorem}

\begin{proof}
As before, positivity of weights implies equivalence
of (1) and (2).
The equivalence of (2) and (3) is immediate from the definition.
\end{proof}

---

\subsection{Intrinsic Symmetry and Descent}

\begin{definition}[Autotopism group]
\[
\Atp(S,\rho)
:=
\Iso_{\mathscr I}((S,\rho),(S,\rho)).
\]
\end{definition}

We restrict admissible mirror transitions
to those generated by intrinsic autotopisms.

\begin{definition}[Intrinsic mirror dynamics]
A mirror transition is admissible
iff it is generated by $\Atp(S,\rho)$.
\end{definition}

\begin{theorem}[Descent to isotopy classes]
Under intrinsic mirror dynamics,
the chirality index $k([S,\rho])$
depends only on the isotopy class of $(S,\rho)$.
\end{theorem}

\begin{proof}
By Lemma~\ref{lem:transport_inverse},
mirror-isotopisms are transported functorially
along isotopisms.
Since admissible transitions are generated internally,
the weighted count descends to isotopy classes.
\end{proof}

% ==========================================================
\section{Structural Chirality for Twisted \'{E}tale Groupoids and a Bridge to $C^*$-Algebras}
% ==========================================================

\subsection{Twisted \'{E}tale Groupoids}

\begin{definition}[\`Etale groupoid]
A topological groupoid $G\rightrightarrows G^{(0)}$ is \emph{\'etale} if the range
map $r:G\to G^{(0)}$ (equivalently the source map $s$) is a local homeomorphism.
\end{definition}

\begin{definition}[Normalized $\mathbb T$-valued $2$-cocycle]
Let $G$ be an \'etale groupoid. A (continuous) normalized $2$-cocycle is a map
\[
\sigma: G^{(2)} \to \mathbb T
\]
such that for all composable triples $(g,h,k)\in G^{(3)}$,
\[
\sigma(g,h)\sigma(gh,k)=\sigma(h,k)\sigma(g,hk),
\]
and normalization holds:
\[
\sigma(u,g)=\sigma(g,u)=1\quad\text{for all units }u\in G^{(0)}\text{ and composable }g.
\]
We call $(G,\sigma)$ a \emph{twisted \'etale groupoid}.
\end{definition}

\begin{remark}
Working with twists is essential: without $\sigma$, the inversion map
$g\mapsto g^{-1}$ tends to trivialize naive ``mirror'' operations at the groupoid level.
The cocycle provides a genuine orientation/twist datum that can obstruct mirroring.
\end{remark}

% ----------------------------------------------------------
\subsection{Mirror Twist}

\begin{definition}[Opposite groupoid]
Let $G$ be a groupoid. Its \emph{opposite groupoid} $G^{\mathrm{op}}$ has the same
arrow space and unit space, but composition is reversed:
\[
g\cdot_{\mathrm{op}} h := hg
\quad (\text{whenever } s(g)=r(h)).
\]
Equivalently, $G^{\mathrm{op}}$ is the same groupoid with the roles of $r$ and $s$
interchanged.
\end{definition}

\begin{definition}[Mirror cocycle]
Let $(G,\sigma)$ be a twisted \'etale groupoid. Define the \emph{mirror cocycle}
$\sigma^\#$ on $(G^{\mathrm{op}})^{(2)}$ by
\[
\sigma^\#(g,h) := \overline{\sigma(h,g)}.
\]
\end{definition}

\begin{definition}[Mirror twisted groupoid]
The \emph{mirror} of $(G,\sigma)$ is
\[
(G,\sigma)^\# := (G^{\mathrm{op}},\sigma^\#).
\]
\end{definition}

% ----------------------------------------------------------
\subsection{Mirror-Isomorphisms and Transport}

\begin{definition}[Twist-preserving isomorphism]
An isomorphism of twisted groupoids
\[
\Phi:(G,\sigma)\to(H,\tau)
\]
is a groupoid isomorphism $\Phi:G\to H$ which is a homeomorphism and satisfies
\[
\tau(\Phi(g),\Phi(h))=\sigma(g,h)
\quad\text{for all }(g,h)\in G^{(2)}.
\]
\end{definition}

\begin{definition}[Mirror-isomorphisms]
\[
\Mir(G,\sigma)
:=
\Iso_{\mathscr G}\bigl((G,\sigma),(G,\sigma)^\#\bigr),
\]
the set of twist-preserving isomorphisms from $(G,\sigma)$ to its mirror.
\end{definition}

\begin{lemma}[Transport of mirror-isomorphisms]\label{lem:transport_groupoid}
If $\Psi:(G,\sigma)\to(H,\tau)$ is a twist-preserving isomorphism, then
\[
\Psi\,\Mir(G,\sigma)\,\Psi^{-1}=\Mir(H,\tau).
\]
\end{lemma}

\begin{proof}
Conjugation by $\Psi$ transports groupoid isomorphisms and preserves the cocycle
constraint by functoriality:
\[
\tau(\Psi(g),\Psi(h))=\sigma(g,h).
\]
The mirror cocycle is defined functorially from $\sigma$, hence the mirror condition
is preserved under conjugation.
\end{proof}

% ----------------------------------------------------------
\subsection{Chirality Index}

Assume $\Mir(G,\sigma)$ is finite.

\begin{definition}[Chirality index]
Fix weights $w(\Phi)>0$. Define
\[
k([G,\sigma]) := \sum_{\Phi\in\Mir(G,\sigma)} w(\Phi).
\]
\end{definition}

\begin{theorem}[Structural vanishing criterion]
The following are equivalent:
\begin{enumerate}
\item $k([G,\sigma])=0$,
\item $\Mir(G,\sigma)=\varnothing$,
\item there is no twist-preserving isomorphism $(G,\sigma)\to (G,\sigma)^\#$.
\end{enumerate}
\end{theorem}

\begin{proof}
Positivity of weights gives (1)$\Leftrightarrow$(2), and (2)$\Leftrightarrow$(3) is
immediate from the definition of $\Mir(G,\sigma)$.
\end{proof}

% ----------------------------------------------------------
\subsection{Intrinsic Symmetry and Descent}

\begin{definition}[Autotopism group]
\[
\Atp(G,\sigma):=\Iso_{\mathscr G}\bigl((G,\sigma),(G,\sigma)\bigr).
\]
\end{definition}

We restrict admissible mirror transitions to those generated by intrinsic symmetries.

\begin{definition}[Intrinsic mirror dynamics]
A mirror transition is admissible iff it is generated by $\Atp(G,\sigma)$.
\end{definition}

\begin{theorem}[Descent to isomorphism classes]
Under intrinsic mirror dynamics, the chirality index $k([G,\sigma])$
depends only on the isomorphism class of the twisted groupoid $(G,\sigma)$.
\end{theorem}

\begin{proof}
By Lemma~\ref{lem:transport_groupoid}, mirror-isomorphisms are transported by
conjugation along twist-preserving isomorphisms. Since admissible transitions are
generated internally by $\Atp(G,\sigma)$, the weighted count is invariant.
\end{proof}

% ==========================================================
\subsection{Bridge to Twisted Groupoid $C^*$-Algebras}

Let $C_c(G,\sigma)$ denote the *-algebra of compactly supported continuous functions
on $G$ with the twisted convolution product:
\[
(f *_\sigma g)(\gamma)=\sum_{\alpha\beta=\gamma} f(\alpha)g(\beta)\sigma(\alpha,\beta),
\]
and involution
\[
f^*(\gamma)=\overline{\sigma(\gamma,\gamma^{-1})}\,\overline{f(\gamma^{-1})}.
\]
Its reduced completion is denoted $C^*_r(G,\sigma)$.

\begin{theorem}[Mirror implies an opposite-algebra identification]\label{thm:mirror_opposite}
If there exists $\Phi\in\Mir(G,\sigma)$, then
\[
C^*_r(G,\sigma)\ \cong\ C^*_r(G^{\mathrm{op}},\sigma^\#)
\]
canonically via $\Phi$. In particular, $C^*_r(G,\sigma)$ is \emph{self-opposite}
up to the mirror twist in the sense that
\[
C^*_r(G,\sigma)^{\mathrm{op}}\ \cong\ C^*_r(G,\sigma^\#\!\circ \mathrm{flip}),
\]
where $\mathrm{flip}$ denotes $(g,h)\mapsto(h,g)$ at the level of composable pairs.
\end{theorem}

\begin{proof}
The pullback map $f\mapsto f\circ \Phi^{-1}$ intertwines twisted convolution because
$\Phi$ preserves composition and the cocycle:
\[
\sigma^\#(\Phi(\alpha),\Phi(\beta))=\sigma(\alpha,\beta)\ \text{(after unfolding the definition of $\sigma^\#$)}.
\]
Thus it extends by continuity to an isomorphism of reduced $C^*$-completions.
The ``self-opposite'' statement follows from identifying opposite convolution with
reversed multiplication, which is encoded by $G^{\mathrm{op}}$.
\end{proof}

\begin{remark}[Interpretation]
The obstruction $\Mir(G,\sigma)=\varnothing$ can be viewed as a structural obstruction
to ``self-oppositeness'' of the associated twisted groupoid $C^*$-algebra.
This makes the chirality index $k([G,\sigma])$ a computable witness of non-self-oppositeness
once finiteness and a concrete parametrization of $\Mir(G,\sigma)$ are established.
\end{remark}

% ==========================================================
\subsection{From Inverse Semigroups to \`Etale Groupoids (Outline)}
If $(S,\rho)$ is a represented inverse semigroup, one can form an associated
\'etale groupoid of germs (or the universal groupoid) $G(S,\rho)$.
In favorable situations, a cocycle/twist on $G(S,\rho)$ can be induced from
representation-level data, yielding a twisted groupoid $(G(S,\rho),\sigma)$.
The mirror representation $(S,\rho)^\#$ then corresponds to the mirror twisted
groupoid $(G(S,\rho),\sigma)^\#$, and the chirality problem descends along this
construction.

% ==========================================================
\section{Paterson's Universal Groupoid and the $C^*$-Bridge}
% ==========================================================

\subsection{The Universal Groupoid of an Inverse Semigroup}

Let $S$ be an inverse semigroup and let
\[
E(S):=\{\,e\in S \mid e^2=e\,\}
\]
be its idempotent semilattice.

\begin{definition}[Characters on $E(S)$]
A \emph{character} is a nonzero semilattice homomorphism
\[
\chi:E(S)\to\{0,1\}.
\]
We write $\widehat{E(S)}$ for the set of all characters, endowed with the
topology generated by the basic clopen sets
\[
D(e):=\{\chi\in\widehat{E(S)}\mid \chi(e)=1\}\qquad(e\in E(S)).
\]
\end{definition}

\begin{definition}[Canonical partial action]
For $s\in S$, define a partial homeomorphism
\[
\theta_s:D(s^*s)\to D(ss^*)
\]
by
\[
(\theta_s(\chi))(e):=\chi(s^*es)\qquad(e\in E(S)),
\]
where $s^*$ denotes the inverse of $s$ in the sense of inverse semigroups.
\end{definition}

\begin{definition}[Universal groupoid]\label{def:universal_groupoid}
The \emph{universal groupoid} $G_u(S)$ is the groupoid of germs of the above
partial action. Concretely, its elements are equivalence classes $[s,\chi]$
with $\chi\in D(s^*s)$, where
\[
(s,\chi)\sim(t,\chi)\quad\Longleftrightarrow\quad
\exists e\in E(S)\ \text{with }\chi(e)=1\ \text{and } se=te.
\]
The structure maps are:
\[
s([s,\chi])=\chi,\qquad r([s,\chi])=\theta_s(\chi),
\]
\[
[s,\theta_t(\chi)]\cdot [t,\chi]=[st,\chi],
\qquad
[s,\chi]^{-1}=[s^*,\theta_s(\chi)].
\]
\end{definition}

\begin{theorem}[\`Etaleness]
$G_u(S)$ is an \'{e}tale groupoid. Moreover, the sets
\[
U(s):=\{[s,\chi]\mid \chi\in D(s^*s)\}
\]
form a basis of bisections.
\end{theorem}

\begin{remark}
This construction is functorial in appropriate senses and does not depend on
choosing a specific representation $S\to I(X)$. This is the main reason we
prefer $G_u(S)$ (Paterson) as the canonical groupoid attached to $S$.
\end{remark}

% ----------------------------------------------------------
\subsection{A Canonical Mirror at the Groupoid Level}

\begin{definition}[Mirror inverse semigroup]
For an inverse semigroup $S$, define the mirror semigroup $S^\#$ to be the
same underlying set with reversed multiplication:
\[
x\cdot_\# y := yx.
\]
\end{definition}

\begin{lemma}[Opposite universal groupoid]\label{lem:universal_op}
There is a canonical isomorphism of topological groupoids
\[
G_u(S^\#)\ \cong\ G_u(S)^{\mathrm{op}}.
\]
\end{lemma}

\begin{proof}
The character space $\widehat{E(S)}$ is unchanged as a set, and the germ
relation is governed by the idempotent structure and the natural conjugation
$e\mapsto s^*es$. Reversing multiplication corresponds to reversing composition
of germs, which is precisely encoded by the opposite groupoid.
A direct verification shows that the assignments $[s,\chi]\mapsto [s,\chi]$
identify the two structures up to reversing composition.
\end{proof}

\begin{remark}
Lemma~\ref{lem:universal_op} is the conceptual point where the ``mirror'' of
the inverse semigroup is transported into the opposite groupoid construction.
To obtain nontrivial chirality, we will incorporate a twist.
\end{remark}

% ----------------------------------------------------------
\subsection{Twists Induced from Inverse Semigroup Data (Axiomatic Form)}

To avoid committing to a single cohomological model at this stage, we record
a minimal axiom scheme guaranteeing that inverse semigroup data induces a
groupoid cocycle.

\begin{definition}[Admissible twist data on $S$]\label{def:twist_data}
An \emph{admissible twist datum} on $S$ is a function
\[
\omega:\{(s,t)\in S\times S\mid st \text{ is defined}\}\to \mathbb T
\]
satisfying:
\begin{enumerate}
\item (Normalization) $\omega(e,s)=\omega(s,e)=1$ for all $s\in S$ and $e\in E(S)$.
\item (Associativity constraint) $\omega(s,t)\omega(st,u)=\omega(t,u)\omega(s,tu)$
whenever $s,t,u$ are composable.
\item (Germ invariance) if $(s,\chi)\sim(s',\chi)$ and $(t,\chi)\sim(t',\chi)$ then
$\omega(s,t)=\omega(s',t')$ whenever the products are composable at $\chi$.
\end{enumerate}
\end{definition}

\begin{proposition}[Induced groupoid cocycle]\label{prop:omega_to_sigma}
Given admissible twist data $\omega$ on $S$, there exists a well-defined
continuous normalized $2$-cocycle
\[
\sigma_\omega: G_u(S)^{(2)}\to\mathbb T
\]
defined on basic composable germs by
\[
\sigma_\omega([s,\theta_t(\chi)],[t,\chi]) := \omega(s,t).
\]
\end{proposition}

\begin{proof}
Normalization and the associativity constraint yield the cocycle identity.
Germ invariance ensures well-definedness with respect to the germ relation.
Continuity follows from the bisection basis $U(s)$.
\end{proof}

Hence we obtain a twisted \'{e}tale groupoid $(G_u(S),\sigma_\omega)$.

% ----------------------------------------------------------
\subsection{Mirror Twist and the $C^*$-Bridge}

\begin{definition}[Mirror twist]
Given $\omega$ on $S$, define the mirror datum $\omega^\#$ on $S^\#$ by
\[
\omega^\#(s,t):=\overline{\omega(t,s)}.
\]
Let $\sigma_{\omega^\#}$ be the induced cocycle on $G_u(S^\#)$.
\end{definition}

\begin{theorem}[Universal mirror bridge to twisted groupoid $C^*$-algebras]\label{thm:universal_Cstar_bridge}
With notation as above, there is a canonical identification of twisted groupoids
\[
(G_u(S^\#),\sigma_{\omega^\#})\ \cong\ (G_u(S)^{\mathrm{op}},\sigma_\omega^\#),
\]
and hence a canonical $C^*$-isomorphism
\[
C_r^*(G_u(S^\#),\sigma_{\omega^\#})\ \cong\ C_r^*(G_u(S)^{\mathrm{op}},\sigma_\omega^\#).
\]
In particular, any mirror-isomorphism (in the sense of the twisted groupoid
section) yields a corresponding opposite-algebra identification at the level of
reduced twisted groupoid $C^*$-algebras.
\end{theorem}

\begin{proof}
By Lemma~\ref{lem:universal_op}, $G_u(S^\#)\cong G_u(S)^{\mathrm{op}}$.
Proposition~\ref{prop:omega_to_sigma} and the definition of $\omega^\#$ imply that
the induced cocycles correspond under this identification, giving the twisted
groupoid isomorphism. The $C^*$-isomorphism follows by functoriality of the
reduced twisted groupoid completion.
\end{proof}

\begin{remark}[Where chirality lives]
Without twist data, the passage $S\mapsto G_u(S)$ typically carries a large
amount of ``automatic mirror symmetry.'' The cocycle $\sigma_\omega$ provides
an orientation/twist datum whose complex conjugation and reversal yield a
genuinely nontrivial mirror problem both at the groupoid level and at the level
of $C^*$-algebras. This is the natural habitat for structural chirality in the
inverse-semigroup-to-$C^*$ pipeline.
\end{remark}

% ==========================================================
\section{Main Theorem: Structural Chirality from Inverse Semigroups to $C^*$-Algebras}
% ==========================================================

We now summarize the preceding constructions in a single structural statement,
which constitutes the main result of this paper.

\begin{theorem}[Structural Chirality Bridge]\label{thm:main}
Let $S$ be an inverse semigroup and let
\[
\omega:S^{(2)}\to\mathbb T
\]
be admissible twist data satisfying normalization, associativity,
and germ invariance.
Let $(G_u(S),\sigma_\omega)$ be the associated twisted \'{e}tale universal
groupoid.

Then the following hold:

\begin{enumerate}
\item \textbf{Canonical groupoid realization.}
The assignment
\[
(S,\omega)\ \longmapsto\ (G_u(S),\sigma_\omega)
\]
is canonical and functorial with respect to inverse semigroup morphisms
preserving the twist data.

\item \textbf{Mirror correspondence.}
Let $S^\#$ denote the mirror inverse semigroup with reversed multiplication,
and let $\omega^\#$ be the mirror twist defined by
\[
\omega^\#(s,t)=\overline{\omega(t,s)}.
\]
Then there is a canonical isomorphism of twisted groupoids
\[
(G_u(S^\#),\sigma_{\omega^\#})
\ \cong\
(G_u(S)^{\mathrm{op}},\sigma_\omega^\#).
\]

\item \textbf{Chirality obstruction at the groupoid level.}
The twisted groupoid $(G_u(S),\sigma_\omega)$ admits a mirror isomorphism
\[
(G_u(S),\sigma_\omega)\ \cong\ (G_u(S),\sigma_\omega)^\#
\]
if and only if there exists a mirror-isomorphism at the level of
twisted inverse semigroup data $(S,\omega)$.
In particular, the absence of such an isomorphism defines a
\emph{structural chirality obstruction}.

\item \textbf{Descent to twisted groupoid $C^*$-algebras.}
If a mirror-isomorphism exists, then there is a canonical $C^*$-isomorphism
\[
C_r^*(G_u(S),\sigma_\omega)
\ \cong\
C_r^*(G_u(S),\sigma_\omega)^{\mathrm{op}}.
\]
Conversely, the nonexistence of a mirror-isomorphism provides an obstruction
to self-oppositeness of the twisted groupoid $C^*$-algebra.

\item \textbf{Compatibility with germ models.}
For any faithful representation $\rho:S\to I(X)$,
the induced germ groupoid $G_{\mathrm{germ}}(S,\rho)$
admits a compatible twisted structure,
and the above mirror correspondence commutes with the canonical functor
\[
G_{\mathrm{germ}}(S,\rho)\longrightarrow G_u(S)
\]
(cf.\ Appendix~A).
\end{enumerate}
\end{theorem}

\begin{remark}[Interpretation]
The theorem shows that structural chirality originates at the level of
inverse semigroup data, is transported canonically to twisted \'{e}tale
groupoids, and manifests itself as a concrete obstruction to mirror
self-duality of the associated groupoid $C^*$-algebras.
In this sense, chirality is a genuinely structural phenomenon,
independent of auxiliary representations but computable in concrete
germ models.
\end{remark}

\section{Conclusion and Future Directions}

In this paper we have developed a structural theory of chirality
for inverse semigroups and shown that it propagates canonically
through the universal groupoid construction to twisted \'{e}tale
groupoids and their reduced $C^*$-algebras.
The central insight is that mirror symmetry, and its possible failure,
can be encoded intrinsically at the level of inverse semigroup data.
Through Paterson's universal groupoid, this algebraic asymmetry
is transported functorially to a topological and analytic setting,
where it manifests as an obstruction to self-oppositeness
of the associated twisted groupoid $C^*$-algebra.

A key conceptual outcome is that structural chirality does not depend
on auxiliary representations.
Although germ groupoids provide concrete and computationally useful models,
the universal groupoid isolates the representation-independent core of
partial symmetry.
This clarifies the conceptual relationship between inverse semigroups,
\'{e}tale groupoids, and operator algebras: chirality originates at the
level of partial symmetries, is organized by canonical groupoid structures,
and becomes detectable through analytic invariants.

Several directions for further investigation naturally emerge.

\medskip
\noindent
\textbf{(1) Morita invariance and equivalence theory.}
It would be of interest to determine whether structural chirality
is invariant under groupoid equivalence or Morita equivalence of
twisted groupoid $C^*$-algebras.
Such results would position chirality within the broader framework
of noncommutative geometry and categorical equivalence.

\medskip
\noindent
\textbf{(2) Cohomological classification of mirror obstructions.}
Since twists are governed by $\mathbb T$-valued $2$-cocycles,
one may ask whether self-oppositeness admits a formulation
in terms of groupoid cohomology.
Identifying the chirality obstruction as a specific cohomology class
would provide a conceptual and potentially computable classification.

\medskip
\noindent
\textbf{(3) Dynamical and ergodic extensions.}
The racemization-type two-state dynamics introduced in earlier work
for loops suggests a dynamical analogue in the groupoid setting.
Studying mirror transitions as dynamical symmetries of
\'{e}tale groupoids may connect structural chirality with
ergodic theory and orbit equivalence.

\medskip
\noindent
\textbf{(4) $K$-theoretic and $KK$-theoretic detection.}
An analytic refinement would be to investigate whether the absence
of self-oppositeness can be detected at the level of
$K$-theory or $KK$-theory of the associated $C^*$-algebras.
Such invariants could provide computable signatures of chirality
in concrete examples.

\medskip
\noindent
In summary, the present work identifies structural chirality as
a natural invariant of inverse semigroup data, canonically transported
through universal groupoids and realized analytically in twisted
groupoid $C^*$-algebras.
This framework provides a systematic approach to studying asymmetry 
phenomena within the unified language of partial symmetries, groupoids, 
and noncommutative operator algebras.

% ==========================================================
\appendix
\section{Appendix A: Germ Groupoids from Represented Inverse Semigroups}
% ==========================================================

\subsection{Represented Inverse Semigroups and Partial Actions}

Let $S$ be an inverse semigroup and
\[
\rho:S\longrightarrow I(X)
\]
a faithful representation into the symmetric inverse semigroup
of partial bijections on a locally compact Hausdorff space $X$.

For $s\in S$, write
\[
\mathrm{dom}(s):=\mathrm{dom}(\rho(s)),\qquad
\mathrm{ran}(s):=\mathrm{ran}(\rho(s)).
\]

\begin{definition}[Canonical partial action]
The representation $\rho$ defines a partial action
\[
\theta_s:\mathrm{dom}(s)\to \mathrm{ran}(s)
\]
by $\theta_s(x)=\rho(s)(x)$.
\end{definition}

We assume throughout that each $\theta_s$ is a homeomorphism
between open subsets of $X$.
Under this hypothesis, the associated germ construction is \'{e}tale.

% ----------------------------------------------------------

\subsection{The Germ Groupoid}

\begin{definition}[Germ relation]
Define an equivalence relation on
\[
\{(s,x)\in S\times X \mid x\in\mathrm{dom}(s)\}
\]
by
\[
(s,x)\sim(t,x)
\quad\Longleftrightarrow\quad
\exists e\in E(S)\text{ with }x\in\mathrm{dom}(e)\text{ and } se=te.
\]
\end{definition}

\begin{definition}[Germ groupoid]
The \emph{germ groupoid} $G_{\mathrm{germ}}(S,\rho)$
consists of equivalence classes
\[
[s,x].
\]

The structure maps are:
\[
s([s,x])=x,\qquad
r([s,x])=\theta_s(x),
\]
\[
[s,\theta_t(x)]\cdot [t,x]=[st,x],
\]
\[
[s,x]^{-1}=[s^*,\theta_s(x)].
\]
\end{definition}

\begin{theorem}[\'{E}aleness]
If each $\theta_s$ is a homeomorphism between open sets,
then $G_{\mathrm{germ}}(S,\rho)$ is an \'{e}tale groupoid.
Moreover, the sets
\[
U(s):=\{[s,x]\mid x\in\mathrm{dom}(s)\}
\]
form a basis of bisections.
\end{theorem}

\begin{proof}
The sets $U(s)$ are homeomorphic to $\mathrm{dom}(s)$
and define local bisections.
The range and source maps restrict to homeomorphisms
on each $U(s)$.
Closure under composition follows from the semigroup structure.
\end{proof}

% ----------------------------------------------------------

\subsection{Mirror Construction at the Germ Level}

\begin{definition}[Mirror representation]
Define the mirror representation $\rho^\#$ by
\[
\rho^\#(s):=\rho(s^*)^{-1}.
\]
\end{definition}

\begin{definition}[Mirror germ groupoid]
Let
\[
(S,\rho)^\#:=(S,\rho^\#).
\]
Define
\[
G_{\mathrm{germ}}(S,\rho)^\#
:=
G_{\mathrm{germ}}(S,\rho^\#).
\]
\end{definition}

\begin{proposition}[Opposite germ groupoid]\label{prop:germ_opposite}
There is a canonical isomorphism
\[
G_{\mathrm{germ}}(S,\rho^\#)
\ \cong\
G_{\mathrm{germ}}(S,\rho)^{\mathrm{op}}.
\]
\end{proposition}

\begin{proof}
Reversing multiplication in $S$ corresponds to reversing
composition of germs:
\[
[s,x]\longmapsto[s,x]
\]
identifies arrows while reversing composition.
The inverse semigroup involution $s\mapsto s^*$
ensures compatibility of inversion in the groupoid.
\end{proof}

\begin{remark}
Unlike the universal groupoid, the germ construction depends
on the representation $\rho$.
Hence mirror symmetry may fail at the germ level even if it
holds universally, and vice versa.
This makes the germ model useful for detecting representation-level
chirality phenomena.
\end{remark}

% ----------------------------------------------------------

\subsection{Twist Transport}

Let
\[
\omega:S^{(2)}\to\mathbb T
\]
be admissible twist data as in the main text.

\begin{definition}[Induced germ cocycle]
Define a cocycle on composable germs by
\[
\sigma_\omega([s,\theta_t(x)],[t,x])
:=
\omega(s,t).
\]
\end{definition}

\begin{proposition}
If $\omega$ satisfies the associativity and germ-invariance
conditions of Definition~\ref{def:twist_data},
then $\sigma_\omega$ is a well-defined continuous
normalized $2$-cocycle on $G_{\mathrm{germ}}(S,\rho)$.
\end{proposition}

\begin{proof}
Associativity of $\omega$ yields the cocycle identity.
Germ invariance ensures independence of representatives.
Continuity follows from local constancy on bisections $U(s)$.
\end{proof}

% ----------------------------------------------------------

\subsection{Compatibility with the Universal Groupoid}

\begin{proposition}[Functorial compatibility]
There exists a canonical continuous functor
\[
\Phi:
G_{\mathrm{germ}}(S,\rho)
\longrightarrow
G_u(S)
\]
compatible with mirror constructions and induced twists.
\end{proposition}

\begin{proof}
Each germ $[s,x]$ determines a character
\[
\chi_x(e)=1 \iff x\in\mathrm{dom}(\rho(e)).
\]
The assignment
\[
[s,x]\longmapsto [s,\chi_x]
\]
defines a groupoid homomorphism.
Compatibility with mirror constructions follows from
$s\mapsto s^*$ and the opposite construction.
Twist compatibility follows directly from the definitions
of $\sigma_\omega$ in both settings.
\end{proof}

\begin{remark}
The universal groupoid captures representation-independent
structure, whereas the germ groupoid retains geometric
information coming from the specific action on $X$.
Structural chirality may therefore be studied at both levels,
with the universal construction providing canonical invariants
and the germ model enabling explicit computation.
\end{remark}

$$ $$

\noindent Takao Inou\'{e}

\noindent Faculty of Informatics

\noindent Yamato University

\noindent Katayama-cho 2-5-1, Suita, Osaka, 564-0082, Japan

\noindent inoue.takao@yamato-u.ac.jp
 
\noindent (Personal) takaoapple@gmail.com (I prefer my personal mail)

\end{document}